\documentclass[leqno,final]{article}

\usepackage{amsmath,amsthm,amssymb,amsfonts}
\usepackage{indentfirst}
\usepackage[notref,notcite]{showkeys}
\usepackage{enumitem}
\usepackage{srctex}

\newcommand{\R}{\mathbb{R}}
\newcommand{\N}{\mathbb{N}}
\renewcommand{\S}{{\mathbb{S}^{N-1}}}

\newcommand{\eps}{\varepsilon}

\newcommand{\Sym}{{\mathcal{S}}}

\newcommand{\tr}{\mathrm{Tr}\,}

\newtheorem{lem}{Lemma}
\newtheorem{prop}{Proposition}
\newtheorem{thm}{Theorem}
\newtheorem{cor}{Corollary}

\theoremstyle{remark}
\newtheorem{rem}{Remark}

\begin{document}

\title{Alexandroff-Bakelman-Pucci estimate and Harnack inequality for
  degenerate/singular fully non-linear elliptic equations}

\author{Cyril Imbert}

\date{\today}

\maketitle

\begin{quote} \footnotesize
  \noindent \textsc{Abstract.}  In this paper, we study fully
  non-linear elliptic equations in non-divergence form which can be
  degenerate or singular when ``the gradient is small''. Typical examples are
  either equations involving the $m$-Laplace operator or
  Bellman-Isaacs equations from stochastic control problems. We
  establish an Alexandroff-Bakelman-Pucci estimate and we prove a
  Harnack inequality for viscosity solutions of such non-linear
  elliptic equations.
\end{quote}
\vspace{5mm}


\medskip

\bigskip

\section{Introduction}

Following the original strategy of Krylov and Safonov \cite{ks1,ks2},
Delarue \cite{delarue} proved by probabilistic methods a Harnack
inequality for quasi-linear elliptic equations of the form
\begin{equation}\label{eq:quasilinear}
- \tr (A(x,u,Du)D^2 u ) + H (x,u,Du) = 0 \, , \quad x \in \Omega
\end{equation}
(where $\Omega$ is a domain of $\R^n$)
in the case where the $n\times n$ matrix $A(x,p)$ can
degenerate. Precisely, he assumes
\begin{eqnarray}\label{assum:delarue}
\Lambda^{-1} \lambda (p) I \le A(x,u,p) \le \Lambda \lambda (p) I  \\
\label{assum:delarue2}
H(x,u,p) \le \Lambda (1 + \lambda (p)) (1+|p|)
\end{eqnarray}
where $\Lambda \ge 1$, $\lambda: \R^n \to \R^+$ is continuous and such
that $\lambda (p) \ge \lambda_F$ if $|p|\ge M_F$. In
\eqref{assum:delarue}, $I$ denotes the identity matrix and the
inequalities are understood in the sense of the usual partial order on
the set of real symmetric matrices. The model example of
\eqref{eq:quasilinear} is the $m$-Laplace equation where $A(x,p) =
|p|^{m-2}$ for some $m>2$.  An important application of the Harnack
inequality is the derivation of a H\"older estimate for the solution
of \eqref{eq:quasilinear}.

In this paper, we generalize this result to the case of fully
non-linear elliptic equations in non-divergence form
\begin{equation}\label{eq:main}
F(x,u,Du, D^2 u ) = 0 \, , \quad x \in \Omega
\end{equation}
which can be either degenerate or singular. We do so by proving first
an Alexandroff-Bakelman-Pucci (ABP for short) estimate.  This is the first main
difference with \cite{delarue} and the first main contribution of this
paper. Important examples of \eqref{eq:main} which are out of the
scope of \cite{delarue} are Bellman-Isaacs equations appearing in the
context of stochastic control problems.  We also generalize and/or
recover results from \cite{dfq-cras,bd} where an ABP estimate and a
Harnack inequality respectively are obtained for
\begin{equation}\label{eq homog}
  F_0(Du,D^2u) + b(x) \cdot Du |Du|^\alpha + cu |u|^\alpha + f_0 (x)=0 \, , \quad x \in \Omega
\end{equation}
where $F$ is positively homogeneous of order $\alpha \in (-1,1)$ (see
Section~\ref{sec:comparison} for precise assumptions). If $\alpha \in
[0,1)$, the equation is degenerate. If $\alpha \in (-1,0]$, the
equation is singular. Even if this equation does not formally enter
into our general framework, we will explain how the results of
\cite{dfq-cras,bd} can be derived from ours.

\paragraph{Known results.}
Krylov and Safonov \cite{ks1,ks2} first proved a Harnack inequality
for second order elliptic equations in non-divergence form with
measurable coefficients. This result is often presented as the
counterpart of the De Giorgi and Nash estimates \cite{degiorgi,nash}
for divergence form equations.

As far as degenerate elliptic equations are concerned, De Giorgi and
Nash estimates were obtained for equations in divergence form and for
degeneracies of $p$-Laplace type.  See for instance \cite{serrin,lsu}.

Krylov and Safonov estimates were obtained by Caffarelli \cite{caff89}
for fully non-linear elliptic equations of the form $F(x,D^2u) =0$
(see also \cite{trudinger,gt}). As explained in \cite{cc95}, a
fondamental tool in this approach is the Alexandroff-Bakelman-Pucci
estimate. Many authors extended these results since then; see for
instance \cite{fok,ks,clv05,qs} and references therein.

To the best of our knowledge and as far as degenerate elliptic
equations in non-divergence form are concerned, the Krylov and Safonov
estimates obtained by Delarue \cite{delarue} are the first ones.

After this work was completed, Birindelli and Demengel \cite{bd}
obtained a Harnack inequality for degenerate elliptic equations of the
form \eqref{eq homog} with $\alpha \in [0,1)$ in dimension
$2$. Reading their interesting paper, we understood that we could
recover (and in fact extend) their results and deal with singular
equations. We will explain how to get the same estimate in any
dimension (see Section~\ref{sec:comparison}). Their work aims at
generalizing the results of D\'avila, Felmer and Quaas \cite{dfq-sing}
where the same elliptic equation is considered but with $\alpha \in
(-1,0]$. Hence, the equation is singular. We also mention that  an ABP
estimate is proved in \cite{dfq-cras} for degenerate and singular
equations.  We will explain that it can be derived from ours;
see Section~\ref{sec:comparison} where our results are compared with
the ones in \cite{bd,dfq-cras}.

\paragraph{Main results.}
Let us now describe a bit more precisely our main results.  
We use the techniques
developed by Caffarelli \cite{caff89} (see also \cite{cc95}) instead
of probability arguments to get, apart from the
Alexandroff-Bakelman-Pucci estimate, a weak Harnack inequality and a
local maximum principle. It is then easy to derive a Harnack
inequality and a H\"older estimate of a solution of \eqref{eq:main}.

First and foremost, we mention that, as in \cite{caff89,delarue}, we
use the notion of viscosity solution \cite{cil92} since the equation
is fully non-linear. We recall that if singular equations of the
form~\eqref{eq homog} are considered, the classical notion of
viscosity solutions must be adapted; see \cite{bd2}.

We next make precise the standing assumptions that the non-linearity $F$ must satisfy. 
Throughout the paper, $\Sym_n$ denotes the space of real symmetric $n \times n$ matrices
and $B_R$ denotes the open ball of radius $R\ge 0$.
\paragraph{Assumption (A).}
\begin{itemize}
\item
$F$ is \emph{continuous} on $\Omega \times \R \times \R^n \setminus B_{M_F} \times \Sym_n$ for
some $M_F\ge 0$;
\item $F$ is \textit{(degenerate) elliptic}, 
\textit{i.e.}  for all $x \in \Omega$, $r \in \R$, $p \in \R^n$ ($p \neq 0$ 
for singular equation) and $X,Y \in 
\Sym_n$, 
$$
X \le Y \Rightarrow F(x,r,p,Y) \le F(x,r,p,X) \, .
$$
\item
$F$ is \textit{proper} i.e. it is non-decreasing with respect to its $r$ variable. 
\end{itemize}
  Our first main result (Theorem~\ref{thm:abp}) is an ABP
estimate for lower semi-continuous super-solutions of \eqref{eq:main}
on a ball $B_d$ where $F$ is \textit{strictly elliptic for ``large
  gradients''}
\begin{equation} \label{strictell}
\left.\begin{array}{r}
X \ge 0 \\
|p| \ge M_F \\
F(x,r,p,X) \ge 0 
\end{array}\right\} \Rightarrow
 - \lambda_F \mathrm{tr} (X) + \sigma (x) |p| + g(x,u) \ge 0
\end{equation}
for some continuous functions $g$ and $\sigma$ and some constants $M_F
\ge 0$, $\lambda_F >0$.  This condition holds true if $F$ satisfies
\eqref{loclip2} but it is more general.  An ABP estimate permits us to
control $\sup_{B_d} u^-$ in terms of $M_\partial=\sup_{\partial B_d}
u^-$ and the $L^n$-norms of $g(x,M_\partial)$ and $\sigma$ appearing in
\eqref{strictell}. In order to get such an estimate, we use the
techniques from \cite{caff89}. As we already mentioned it
in \cite{convexe}, the ABP estimate that we are able to obtain differs
slightly from classical ones in the sense that we can prove it under
a weaker condition than \eqref{loclip2}; moreover, the
super-solution is only lower semi-continuous. We recall that 
this is an a priori estimate: structure conditions
ensuring the uniqueness of the solution are not required. 
We finally mention that when the equation is strictly elliptic
($M_F=0$), we recover the classical ABP estimate. 

Our second main result (Corollary~\ref{cor:hi}) is a Harnack inequality
for \eqref{eq:main}. This inequality is a consequence of a weak Harnack
inequality and a local maximum principle proved by generalizing in 
an appropriate way \eqref{assum:delarue} and \eqref{assum:delarue2}.
In view of \eqref{assum:delarue}, one can consider the quasilinear
equation~\eqref{eq:quasilinear} where $A$ and $H$ are replaced with 
$$
\tilde{A} (x,u,Du) = \frac1{\lambda(Du)}{A(x,u,Du)} \quad \text{ and }
\quad \tilde{H}(x,u,Du) = \frac1{\lambda(Du)}{H(x,u,Du)} \, .
$$  
Hence, the new quasi-linear equation is uniformly elliptic. However,
the first order term is, in this case, eventually singular and
\eqref{assum:delarue} can be seen as an assumption concerning the
first order term. In the case of the $m$-Laplace equation, $\lambda
(p)=|z|^{m-2}$ and $H$ has therefore a polynomial growth of order
$m-1$.  Assumptions~\eqref{assum:delarue},
\eqref{assum:delarue2} are replaced with
\begin{eqnarray}
\label{loclip2}
\left. \begin{array}{r} |p| \ge M_F \\
F(x,u,p,X) \ge 0 \end{array} \right\}
\Rightarrow  
 \mathcal{M}^+ (X) + \sigma (x)|p| + \gamma_F u + f(x) \ge 0 \, ,\\
\label{loclip3}
\left. \begin{array}{r} |p| \ge M_F \\ 
F(x,u,p,X) \le 0 \end{array} \right\}
\Rightarrow  \mathcal{M}^- (X) - \sigma (x)|p| + \gamma_F u - f(x) \le 0
\end{eqnarray}
where $\sigma,f: \overline{B} \to \R$ are continuous and $M_F$ and
$\gamma_F$ are non-negative constants. It is important to remark that
if $F$ satisfies \eqref{loclip2}, \eqref{loclip3}, then it can be
degenerate or singular and it can have a superlinear growth in $p$. 

An important consequence of the Harnack
inequality is the H\"older regularity of solutions of \eqref{eq:main} 
(see Theorem~\ref{cor:holder}). As far as the regularity of solutions of 
\eqref{eq:main} is concerned, we notice that by assuming \eqref{loclip2}
and \eqref{loclip3}, we cannot expect  more than Lipschitz
continuity. Indeed, by making such an assumption, 
we somehow forget about all small gradients and we cannot expect these 
small gradients to be regular. 
We also point out that it is easier to prove the uniqueness
of a H\"older continuous function than to prove a strong comparison
result between discontinuous viscosity sub- and super-solutions 
(which is the classical way to get uniqueness of viscosity solutions 
\cite{cil92}).
To finish with, we shed light on the fact that, as for
 the ABP estimate, we recover the Harnack inequality of \cite{caff89}
in the strictly elliptic case ($M_F=0$).

\paragraph{Extensions.}
We will explain how to deal with non-linearities, after redefining
them if necessary, growing quadratically with respect to the
gradient. Precisely, \eqref{loclip2} and \eqref{loclip3} are replaced
with
\begin{eqnarray}
\label{loclip2-quad}
\left.\begin{array}{r} |p| \ge M_F  \\ F(x,u,p,X) \ge 0 \end{array}\right\}
\Rightarrow  \mathcal{M}^+ (X) + \sigma (x)|p| + \sigma_2 |p|^2 + \gamma_F u + f(x) \ge 0 \, ,
\\
\label{loclip3-quad}
\left.\begin{array}{r} |p| \ge M_F \\ F(x,u,p,X) \le 0 \end{array}\right\}
\Rightarrow
 \mathcal{M}^- (X) - \sigma (x)|p| - \sigma_2 |p|^2 + \gamma_F u - f(x) \le 0  
\end{eqnarray}
where $\sigma,f: \overline{B} \to \R$ are continuous and $M_F, \sigma_2$ and
$\gamma_F$ are non-negative constants.
In this case, it is known \cite{trudinger,kt}
that it is not possible to get a weak Harnack inequality which does
not depend on the $L^\infty$-norm of the solution. See
Section~\ref{sec:extensions} for more details and comments.

As far as extensions of these results are concerned, we would like to
mention next that we could have used $L^p$-viscosity solutions
\cite{ccks} instead of classical continuous viscosity solutions in
order to be able to deal with discontinuous coefficients. We chose not
to do so in order to avoid technicalities but we think that this can
be done. We also mention that it is sometimes more difficult to get a
classical ABP estimate when using this notion of solution; for
instance in \cite{ks}, the ABP estimate does not involve the contact
set of the function.

We also mention that the parabolic case will be addressed in a future
work.

\paragraph{Additional comments.} Assumption~\eqref{strictell}
permits to take into account non-linearity growing linearly with
respect to the gradient. Such an assumption appears in \cite{trudinger2} 
where Trudin\-ger proved that 
strong solutions satisfy a weak Harnack inequality for such non-linearities
if $\sigma$ is sufficiently integrable. This result has been generalized to 
viscosity solutions since then; see for instance \cite{fokthesis,ksbis}.

We recall that it is possible to use the techniques introduced
in \cite{il90} in order to prove the H\"older regularity of viscosity
solutions much more easily.
But the estimate of the H\"older constant depends in this case on the
modulus of continuity of the coefficients of the equation.

\paragraph{Organization of the article.} The paper is organized as
follows.  In Section~\ref{sec:prelim}, we construct a barrier function
that will be used when proving the Harnack inequality. We also recall
the definition of two Pucci operators. In Section~\ref{sec:abp}, we
establish an ABP estimate. In Section~\ref{sec:harnack}, we
successively prove a weak Harnack inequality and a local maximum
principle. We also derive from these two results a Harnack inequality.
In Section~\ref{sec:extensions}, we explain how to deal with elliptic
equations with quadratic dependence on the gradient.  As applications
of our results, we generalize and/or recover some results from
\cite{bd,dfq-cras} in Section~\ref{sec:comparison}.
Section~\ref{sec:proofs} is dedicated to proofs of our main results.
Appendix~\ref{app:additional} is added for the sake of completeness of
proofs and for the reader's convenience. We give in
Appendix~\ref{app:additional} detailed proofs of results which can be
easily derived from classical ones.

 \paragraph{Notation.} A ball of radius $r$ centered at $x$ is
denoted by $B(x,r)$ or $B_r (x)$. If $x=0$, we simply write $B_r$. 
$\omega_n$ denotes the volume of the unit ball. 
The hypercube $\Pi_{i=1}^n (x_i -r/2,x_i+r/2)$ is denoted by $Q_r (x)$. 
If $x=0$, we simply write $Q_r$. 

Given a vector $a \neq 0$, $\hat a$ denotes $a / |a|$.  $I$ denotes
the identity matrix. The set of real symmetric $n \times n$ matrices
is denoted by $\mathcal{S}_n$.

A constant is universal if it only depends on $n$ (dimension), $q$
(constant greater than $n$ fixed in all the paper), $\lambda_F$ and
$\Lambda_F$ (ellipticity constants).

Given a lower semi-continuous function $u$, $D^{2,-}u (x)$ (resp. $\bar{D}^{2,-} u (x)$) 
denotes the set of all subjets (resp. limiting subjets) of $u$ at point $x$. 
See \cite{cil92} for definitions. 

\paragraph{Acknowledgments.} We are very grateful to Delarue for
bringing our attention to this problem and for the fruitful
discussions we had together.  We also would like to thank
Capuzzo-Dolcetta, D\'avila, Felmer and Quaas for sending us their
preprints and for their interest in our work and useful comments.  In
particular, the important remarks sent to us by D\'avila, Felmer and
Quaas permit us to improve the first version of this paper and to
improve the results of Birindelli and Demengel.

\section{Preliminaries}
\label{sec:prelim}

\paragraph{Pucci operators.}

We recall the definition of two important second order non-linear elliptic 
operators. 
For all $M \in \mathcal{S}_n$, we define
\begin{eqnarray*}
\mathcal{M}^+ (M) = \sup_{A \in \mathcal{A}_{\lambda_F, \Lambda_F}} ( - \tr (A M)) \\
\mathcal{M}^- (M) = \inf_{A \in \mathcal{A}_{\lambda_F, \Lambda_F}} ( - \tr (A M))
\end{eqnarray*}
where $\mathcal{A}_{\lambda_F , \Lambda_F} = \{ A \in \mathcal{S}_n: \lambda_F I
\le A  \le \Lambda_F I \}$. We will refer to these operators 
as the \emph{maximal} and \emph{minimal Pucci operators}. 
Remark that $\mathcal{M}^+$
is subadditive. More precisely, it is the support function of the set 
$-\mathcal{A}_{\lambda_F,\Lambda_F}$. We will also use the fact that $\mathcal{M}^- (M)
= - \mathcal{M}^+ (-M)$.

\paragraph{Construction of a barrier.}

We now construct a barrier that will be used when proving the (weak) 
Harnack inequality. 
\begin{lem}[Construction of a barrier]    \label{lem:barrier}
Given a constant $\eps_0>0$, there exists a smooth function 
$\varphi: \R^n \to \R$, a universal constant $M_{\mathrm{B}}>1$ and constants  $C_{\mathrm{B}} >0$, 
$R,r>0$ (with $R\ge  (3r/2)\sqrt{n}$)
depending only on the dimension $n$, $\lambda_F$, $\Lambda_F$ and $\eps_0$, 
such that 
\begin{eqnarray}
\varphi \ge 0 &&\text{ in } \R^n \setminus B_R
\label{barrier:lb infinity}\\ 
\varphi \le -2 &&\text{ in } Q_{3r} \label{barrier:local ub}\\
\varphi \ge -M_{\mathrm{B}} &&\text{ in } \R^n \label{barrier:lb}\\ 
\nonumber &&\\\label{barrier:gradient}
| D \varphi | \le \eps_0 && \text{ in } \R^n \\ 
\mathcal{M}^- \varphi  + C_{\mathrm{B}} \xi \ge 0 &&\text{ in } \R^n 
\label{barrier:semi concave}
\end{eqnarray}
where $\xi : \R^n \to [0,1]$ is a continuous function supported
in $\bar{Q}_r$.  
\end{lem}
\begin{rem} We recall that this barrier function will be used to prove the
weak Harnack inequality. 
At first glance, it is not clear why we need to construct
a function $\varphi$ such that $\mathcal{M}^- \varphi \ge 0$ on $Q_r$
and $\varphi \le -2$ on $Q_{3r}$. This will be clearer when applying the
cube decomposition in order to estimate the volume of all the level sets
 (and not only one) of a super-solution. And we choose $R \ge (3r/2)\sqrt{n}$ in order that 
$Q_{3r} \subset B_R$. 
\end{rem}
\begin{proof}
We follow \cite{cc95} by choosing $\varphi$ under the following
form for $x \notin B_r$
$$
\varphi (x) = M_1 - M_2 |x|^{-\alpha }
$$
where $\alpha>0$ will be chosen later and $M_1, M_2>0$ have to be chosen such that 
\eqref{barrier:lb infinity}, \eqref{barrier:local ub} and \eqref{barrier:gradient} 
hold true for $x \notin B_r$ (with $R \ge (3r/2) \sqrt{n}$). It is enough to impose
\begin{eqnarray*}
M_2 &\le& M_1 R^\alpha \, , \\
((3r/2)\sqrt{n})^\alpha (M_1+2)&\le& M_2 \, , \\
M_2 &\le &\eps_0 \frac{r^{\alpha+1}}\alpha \, 
\end{eqnarray*} 
or equivalently
$$
((3r/2)\sqrt{n})^\alpha (M_1+2) \le M_2 \le \min ( M_1 R^\alpha, \eps_0 r^{\alpha +1}/\alpha) \, .
$$
One can choose $M_2$ and $M_1$ so that they satisfy the previous condition 
if and only if 
$$
2 \frac{((3r/2)\sqrt{n})^\alpha}{R^\alpha -((3r/2)\sqrt{n})^\alpha}
\le M_1 \le \frac{\eps_0}{\alpha ((3/2)\sqrt{n})^\alpha} r -2 \, .
$$
Hence, we choose $R = q (3r/2)\sqrt{n}$ with $q>1$ and $r>0$ satisfying
$$
\frac2{q^\alpha -1} \le   \frac{\eps_0}{\alpha ((3/2)\sqrt{n})^\alpha} r -2 \, .
$$
It is now enough to choose $q>1$ such that $\frac2{q^\alpha -1} \le 1$ and 
$r$ such that 
$$\frac{\eps_0}{\alpha ((3/2)\sqrt{n})^\alpha} r \ge 3 \, .
$$ 

We now choose $\alpha >0$ so that \eqref{barrier:semi concave} holds true. If $x \notin B_r$, we have
\begin{eqnarray*}
\mathcal{M}^- (D^2 \varphi (x)) &=& -\alpha M_2 |x|^{-(\alpha+2)} (\Lambda_F (n-1)  - \lambda_F (\alpha+1)) \, .
\end{eqnarray*}
Hence it is enough to choose $\alpha > \max(0, \frac{\Lambda_F}{\lambda_F} (n-1) -1)$ to conclude. 

It is next easy to extend $\varphi$ on $\R^n$ such that
\eqref{barrier:local ub} and \eqref{barrier:gradient} remain true and
\eqref{barrier:lb} is satisfied too for some universal constant
$M_B>1$.  Indeed, we have outside $B_r$
$$
\varphi \ge M_1 -M_2 r^{-\alpha} \ge 2 \frac{1}{q^\alpha -1} -
\frac{\eps_0 r}{\alpha} \, .
$$ 
It is now enough to remark that $q$ and $\eps_0 r$ can be choosen
universal and we also saw above that $\alpha$ can be chosen universal
too. Hence $M_B$ can be chosen universal.
\end{proof}

\paragraph{Rescaling solutions.}
We will have to rescale sub- or super-solutions several times. We need
to know how non-linearities are rescaled in order, for instance, to
determine if these new $F$'s satisfy assumptions.
\begin{lem}[Rescaling solutions]\label{lem:rescaled}
  Given $R_0 >0$, $t_0 >0$ and $x_0 \in \R^n$, let $u$ be a
  super-solution of $F$ on $Q_{t_0 R_0}(x_0)$.  Consider the linear
  map $T:Q_{R_0} \to Q_{t_0 R_0}(x_0)$ defined by $T(y) = x_0 + t_0
  y$.  Then the scaled solution $u_s (y) = \frac1{M_0} u (T(y))$ is a
  super-solution of $F_s = 0$ in $Q_{R_0}$ with
$$
F_s (y,v,q,Y) = \frac{t_0^2}{M_0} F (x_0+t_0 y, M_0 v, t_0^{-1} M_0 q, t_0^{-2} M_0 Y) \, .
$$

If $F$ satisfies \eqref{loclip2} (resp. \eqref{loclip3}), then $F_s$ 
satisfies \eqref{loclip2} (resp. \eqref{loclip3}) with constants $M_s,\gamma_s$ 
and functions  $\sigma_s$ and $f_s$
$$
M_s = \frac{t_0 M_F}{M_0} , \quad 
\gamma_s = t_0^2 \gamma_F, \quad
\sigma_s = t_0 . \sigma \circ T , \quad 
f_s = \frac{t_0^2}{ M_0} f \circ T
 \, .
$$ 
In particular, 
$$
\|f_s\|_{L^n (Q_{R_0})} = \frac{t_0}{M_0}\|f \|_{L^n (Q_{t_0 R_0(x_0)})} \, , \qquad
\|\sigma_s\|_{L^q (Q_{R_0})} = t_0^{1-\frac{n}q}\|\sigma \|_{L^q (Q_{t_0 R_0} (x_0))} \, .
$$
\end{lem}

\section{An ABP estimate}
\label{sec:abp}

As explained in the introduction, we can prove an ABP estimate as soon
as the non-linearity $F$ satisfies a strict ellipticity condition
``for large gradients''. We must also prescribe a growth condition
with respect to first order terms. We thus assume that $F$ satisfies
\eqref{strictell}.  Our first main result is the following theorem.
\begin{thm}[ABP estimate]\label{thm:abp}
Consider a non-linearity $F$ which satisfies {\rm (A)} and \eqref{strictell}.
Let $u$ be a (lsc) super-solution of \eqref{eq:main} in $B_d$. 
Then 
\begin{equation}\label{alextim} 
\sup_{B_d} u^- \le   \sup_{\partial B_d} u^-  + C d \left(M_F +  \left(\int_{B_d \cap \{ u +
  M_\partial = \Gamma  (u)\}} (f^+)^n\right)^{1/n}\right)
\end{equation}
where $M_\partial=  \sup_{\partial B_d} u^-$, $\Gamma (u) $ is the
convex hull of $\min (u+M_\partial,0)$ extended by $0$ on $B_{2d}$, 
 $f(x)=g(x,-M_\partial)$ and $C$ is a constant (only) depending on
$\|\sigma\|_{L^n (B_d)}$, $n$ and $\lambda_F$.   
\end{thm}
\begin{rem}
Remark that when the equation is not degenerate ($M_F=0$), Eq.~\eqref{alextim}
corresponds to the classical ABP estimate. 
\end{rem}
\begin{rem}
The constant $C$ equals  $3e^{C_{\mathrm{ABP}} (1+ ||\sigma||_{L^n(B_d)}^n )}$ 
where $C_{\mathrm{ABP}} = \frac{n 2^{n-2}}{\omega_n \lambda_F^n}$.
\end{rem}
\begin{proof}[Sketch of proof]
The proof follows the ideas of \cite{cc95,convexe}. The key lemma is the following one. 
\begin{lem}\label{lem:1}
The function $\Gamma (u)$ is $C^{1,1}$ on $\mathcal{B} = \{ x \in B_d : |D \Gamma (u) (x)| \ge M_F \}$. 
\end{lem}
\begin{rem}\label{rem:definition de B}
Remark that before knowing that $\Gamma (u)$ is $C^{1,1}$, $D \Gamma (u)$ is not uniquely determined. 
Hence $\mathcal{B}$ should be first defined as follows
$$
\mathcal{B} = \{ x \in B_d : \forall (p,A) \in D^{2,-} \Gamma (u) (x), |p| \ge M_F \} \, .
$$
\end{rem}
Lemma~\ref{lem:1} is proved together with  
\begin{lem}\label{lem:2}
The Hessian of $\Gamma (u)$ satisfies on $\mathcal{B}$ the following properties
\begin{enumerate}
\item $D^2 \Gamma (u) = 0$ a.e. in  $\mathcal{B} \setminus \{ u +
  M_\partial = \Gamma  (u)\}$ ; 
\item $D^2 \Gamma (u) (x) \le \lambda_F^{-1} \big\{\sigma (x) |D
  \Gamma (u) (x)| + f^+ (x)\big\} I$ a.e. in $\mathcal{B} \cap \{ u +
  M_\partial = \Gamma (u)\}$.
\end{enumerate}
\end{lem}
Proofs of these two lemmata can be adapted from the classical ones by
remarking that points $x_i$ called by $x \in \mathcal{B} $ when
computing the convex hull $\Gamma (u)$ (see
Proposition~\ref{prop:combinaison} in Appendix~\ref{app:additional}) satisfy $D\Gamma (u)(x_i) =
D\Gamma (u)(x)$.  In particular, $x_i \in \mathcal{B}$, \textit{i.e.}
$|D \Gamma (u)(x_i)| \ge M_F$ and consequently \eqref{strictell} can
be used.  The reader is referred to Appendix~\ref{app:additional}
where detailed proofs are given for his convenience.
\begin{lem}\label{lem:3}
The following inclusion holds true
\begin{equation}\label{alextim-weak} 
B_{M/(3d)} (0) \setminus B_{M_F} (0 ) \subset D \Gamma (u) (\mathcal{B}) \, .  
\end{equation}
where $M$ denotes $(\sup_{B_d} u^- -\sup_{\partial B_d} u^-)^+$ and 
$\mathcal{B} = \{ x \in B_d : |D \Gamma (u) (x)| \ge M_F \}$. 
\end{lem}
\begin{proof}
This lemma is a consequence of the classical fact
$$
B_{M/(3d)} (0) \subset D \Gamma (u) (B_d) \, .
$$
\end{proof}
From now on, we assume without loss of generality that $M/(3d) \ge M_F$.
We then use Lemma~\ref{lem:1} in order to apply the area formula (see
\cite[Theorem~3.2.5]{federer} and Remark~\ref{rem:areaformula} below) to the
Lipschitz map $D \Gamma (u) : \mathcal{B} \to \R^n$  and to the function
$g(p)= (|p|^{n/(n-1)} + \mu^{n/(n-1)})^{(1-n)}$ for some positive real number $\mu$ to be fixed later.
$$
\int_{D \Gamma (u)(\mathcal{B})} g (p) dp = \int_{\mathcal{B}} g(D \Gamma (u)) \; \mathrm{det} D^2 \Gamma (u) \, . 
$$
On one hand, we can use Lemmata~\ref{lem:2} and \ref{lem:3}  in order to get
\begin{eqnarray*}
\int_{B_{M/(3d)} (0) \setminus B_{M_F}(0)} g(p) dp 
& \le &\int_{D \Gamma (u)(\mathcal{B})} g (p) dp \\
& \le & \int_{\mathcal{B}} g(D \Gamma (u)) \; \mathrm{det} D^2 \Gamma (u) \\
& \le & \frac1{\lambda_F^{n}} \int_{\mathcal{B} \cap \{ u + M_\partial = \Gamma (u) \} } 
g(D \Gamma (u)) (\sigma  |D \Gamma (u) | + f^+ )^n  \\
& \le & \frac{1}{\lambda_F^{n}} \int_{\mathcal{B} \cap \{ u + M_\partial = \Gamma (u) \} } 
(|\sigma|^n + \mu^{-n} (f^+)^n) \, .
\end{eqnarray*}
If now one chooses $\mu$ such that $\mu^n =  \int_{\mathcal{B} \cap \{ u + M_\partial = \Gamma (u) \} } (f^+)^n$,
we obtain from the inequality $g(p) \ge 2^{2-n} (|p|^n + \mu^n)^{-1}$ the following estimate
\begin{eqnarray*}
\frac{2^{2-n}}n \omega_n \ln \frac{(M/(3d))^n + \mu^n}{(M_F)^n + \mu^n} & =& 
2^{2-n} \omega_n \int_{M_F}^{M/d} \frac{r^{n-1} dr}{r^n + \mu^n} \\
&\le& \int_{B_{M/d} (0) \setminus B_{M_F}(0)} g(p) dp \\
& \le &  \lambda_F^{-n}(1+||\sigma||_n^n) 
\end{eqnarray*}
where $\omega_n$ denotes the volume of the unit ball. It is now easy to get \eqref{alextim}.
\end{proof}
\begin{rem}\label{rem:abp-structure}
We see from the previous proof 
that Assumptions~\rm{(A)} and \eqref{strictell} on $F$
are important in order to get the following property
\begin{equation}\label{eq:abp-structure}
\forall (p,A) \in D^{2,-} u (x) \, : \quad
\left.\begin{array}{r} u (x) \le 0 \\ A \ge 0 \\ |p| \ge M_F \end{array}\right\} 
 \Rightarrow \lambda_F \tr A \le \sigma (x) |p| + f(x)  \, .
\end{equation}
As a matter of fact, the previous piece of information is the relevant
one in order to get \eqref{alextim}. Indeed, in Lemma~\ref{lem:2}, the second
estimate can be rewritten as follows
$$
\lambda_F D^2 \Gamma (u)(x) \le \{ \sigma (x) |D \Gamma (u)| + f (x) \} I \, .
$$
\end{rem}
\begin{rem}\label{rem:areaformula}
The area formula in \cite{federer} is stated for maps $G: \R^n \to \R^n$ 
that are Lipschitz continuous on $\R^n$ (in our case). However, the result still 
holds true if $G$ is only Lipschitz continuous on $\mathcal{B}$ since it is always possible to extend it
in a Lipschitz map $\tilde{G}$ on $\R^n$ with $G=\tilde{G}$ on $\mathcal{B}$.
\end{rem}

\section{Harnack inequality}
\label{sec:harnack}

In this section, we explain how to derive a Harnack inequality from
the ABP estimate. As usual, we obtain it 
by deriving on one hand a weak Harnack inequality
and on the other hand a local maximum principle for the fully
nonlinear equation~\eqref{eq:main}.

In order to get a weak Harnack inequality and a local maximum principle
respectively, Condition~\eqref{strictell} is strengthened by assuming
\eqref{loclip2} and \eqref{loclip3} respectively. 

The Harnack inequality is obtained as a combination of the weak Harnack
inequality and the local maximum principle. Here are precise statements.
\begin{thm}[Weak Harnack inequality]\label{thm:whi}
  Given $q>n$ and a non-linearity $F$ satisfying {\rm (A)} and
  \eqref{loclip2} for some continuous functions $f$ and $\sigma$ in
  $Q_1$, consider a non-negative super-solution $u$ of \eqref{eq:main}
  in $Q_1$.  Then
\begin{equation}\label{eq:whi}
\|u\|_{L^{p_0} (Q_{1/4})} \le C ( \inf_{Q_{1/2}} u + \max (M_F, \|f\|_{L^n(Q_1)}) ) 
\end{equation}
where $p_0>0$ is universal and $C$ (only) depends on 
$n$, $q$, $\lambda_F,\Lambda_F$, $\gamma_F$ and $\|\sigma\|_{L^q (Q_1)}$. 
\end{thm}
\begin{thm}[Local maximum principle]\label{thm:lmp}
  Given $q>n$ and a non-linearity $F$ satisfying {\rm (A)} and
  \eqref{loclip3} for some continuous functions $f$ and $\sigma$ on
  $Q_1$, consider a sub-solution $u$ of \eqref{eq:main} in $Q_1$.
  Then for any $p>0$,
\begin{equation}\label{eq:lmp}
\sup_{Q_{1/4}} u \le C(p) ( \| u^+\|_{L^p (Q_{1/2})} + \max (M_F,  \|f\|_{L^n(Q_1)} ) ) 
\end{equation}
where $C(p)$ is a constant (only) depending on 
$n$, $q$, $\lambda_F, \Lambda_F$, $\gamma_F$, $\|\sigma\|_{L^q (Q_1)}$  and $p$.
\end{thm}
Combining these two results, we obtain the second main result of this
paper.
\begin{cor}[Harnack inequality]\label{cor:hi}
  Given $q>n$ and a non-linearity $F$ satisfying {\rm (A)},
  \eqref{loclip2} and \eqref{loclip3} for some continuous functions
  $f$ and $\sigma$ on $Q_1$, consider a non-negative solution $u$ of
  \eqref{eq:main} in $Q_1$.  Then
\begin{equation}\label{eq:hi}
  \sup_{Q_{1/2}} u  \le C ( \inf_{Q_{1/2}} u + \max (M_F, \|f\|_{L^n(Q_1)}) ) 
\end{equation}
where $C$ is a constant (only) depending on $n$,
$q$, $\lambda_F, \Lambda_F$,$\gamma_F$    and $\|\sigma\|_{L^q (Q_1)}$.
\end{cor}
An important consequence of Corollary~\ref{cor:hi} is the following
regularity result.
\begin{cor}[Interior H\"older regularity]\label{cor:holder}
  Given $q>n$ and a non-linearity $F$ satisfying {\rm (A)},
  \eqref{loclip2} and \eqref{loclip3} for some continuous functions
  $f$ and $\sigma$ on $Q_1$, consider a solution $u$ of
  \eqref{eq:main} in $Q_1$.  Then $u$ is $\alpha$-H\"older continuous
  on $Q_{\frac12}$ and
\begin{equation}\label{eq:holder}
\sup_{\stackrel{x, y \in Q_{\frac12}}{ x \neq y}}
\frac{|u(x)-u(y)|}{|x-y|^\alpha} \le C_\alpha (\|u\|_{L^\infty (Q_1)}
+ \max(M_F,\|f\|_{L^n (Q_1)} + \gamma_F \|u\|_{L^\infty (Q_1)}))
\end{equation}
where $\alpha$ and $C_\alpha$ depend (only) on $n$, $q$, $\lambda_F, \Lambda_F$,
$\gamma_F$ and $\|\sigma\|_{L^q (Q_1)}$.
\end{cor}

\section{Quadratic growth in $Du$}
\label{sec:extensions}

In this section, we extend the results of the previous section to
elliptic equations with a first order term (after changing the
original equation if necessary; see the Introduction) which can grow
quadratically with respect to the gradient. Precisely, \eqref{loclip2}
and \eqref{loclip3} are replaced with \eqref{loclip2-quad} and
\eqref{loclip3-quad}.

Through a Cole-Hopf transform, an immediate consequence of
Theorems~\ref{thm:whi} and \ref{thm:lmp} are the following
results. 
\begin{thm}[Weak Harnack inequality] \label{thm:wh-quad}
  Given $q>n$ and a non-linearity $F$ satisfying {\rm (A)} and
  \eqref{loclip2-quad} for some continuous functions $f$ and $\sigma$ in
  $Q_1$, consider a non-negative super-solution $u$ of \eqref{eq:main}
  in $Q_1$.  Then
\begin{equation}\label{eq:whi-quad}
\|u\|_{L^{p_0} (Q_{1/4})} \le C ( \inf_{Q_{1/2}} u + \max (M_F, \|f\|_{L^n(Q_1)}) ) 
\end{equation}
where $p_0>0$ is universal and $C$ (only) depends on $\|u\|_{L^\infty(Q_1)}$, 
$n$, $q$, $\lambda_F,\Lambda_F$, $\gamma_F$ and $\|\sigma\|_{L^q (Q_1)}$. 
\end{thm}
\begin{rem}
  As explained in \cite{trudinger,kt}, one cannot expect to get weak
  Harnack inequality for such non-linearities with a constant $C>0$
  which does not depend on a bound on $u$. 
\end{rem}
\begin{rem}
The constant $C$ can be
  written
$$
C = C_0  \frac{\frac{\sigma_2 \|u\|_{L^\infty (Q_1)} }{\lambda_F}}%
{ 1 - e^{-\frac{\sigma_2 \|u\|_{L^\infty (Q_1)} }{\lambda_F}}}
$$
where  $C_0$ (only) depends on 
$n$, $q$, $\lambda_F,\Lambda_F$, $\gamma_F$ and $\|\sigma\|_{L^q (Q_1)}$.
\end{rem}
\begin{thm}[Local maximum principle]\label{thm:lmp-quad}
  Given $q>n$ and a non-linearity $F$ satisfying {\rm (A)} and
  \eqref{loclip3-quad} for some continuous functions $f$ and $\sigma$ on
  $Q_1$, consider a sub-solution $u$ of \eqref{eq:main} in $Q_1$.
  Then for any $p>0$,
\begin{equation}\label{eq:lmp-quad}
\sup_{Q_{1/4}} u \le C ( \| u^+\|_{L^p (Q_{1/2})} + \max (M_F,  \|f\|_{L^n(Q_1)} ) ) 
\end{equation}
where $C$ (only) depends on $\|u\|_{L^\infty (Q_1)}$, 
$n$, $q$, $\lambda_F, \Lambda_F$, $\gamma_F$, $\|\sigma\|_{L^q (Q_1)}$  and $p$.
\end{thm}
\begin{rem}
The constant $C$ can be
  written
$$
C = C_0  \frac{\frac{\sigma_2 \|u\|_{L^\infty (Q_1)} }{\lambda_F}}%
{ 1 - e^{-\frac{\sigma_2 \|u\|_{L^\infty (Q_1)} }{\lambda_F}}}
$$
where  $C_0$ (only) depends on 
$n$, $q$, $\lambda_F,\Lambda_F$, $\gamma_F$, $\|\sigma\|_{L^q (Q_1)}$ and $p$.
\end{rem}
It is now easy to derive a Harnack inequality and an interior H\"older estimate. 
\begin{cor}[Harnack inequality]\label{cor:hi-quad}
  Given $q>n$ and a non-linearity $F$ satisfying {\rm (A)},
  \eqref{loclip2-quad} and \eqref{loclip3-quad} for some continuous functions
  $f$ and $\sigma$ on $Q_1$, consider a non-negative solution $u$ of
  \eqref{eq:main} in $Q_1$.  Then
\begin{equation}\label{eq:hi-quad}
  \sup_{Q_{1/2}} u  \le C ( \inf_{Q_{1/2}} u + \max (M_F, \|f\|_{L^n(Q_1)}) ) 
\end{equation}
where $C$ (only) depends on $\|u\|_{L^\infty(Q_1)}$, $n$,
$q$, $\lambda_F, \Lambda_F$,$\gamma_F$    and $\|\sigma\|_{L^q (Q_1)}$.
\end{cor}
\begin{cor}[Interior H\"older regularity]\label{cor:holder-quad}
  Given $q>n$ and a non-linearity $F$ satisfying {\rm (A)},
  \eqref{loclip2-quad} and \eqref{loclip3-quad} for some continuous functions
  $f$ and $\sigma$ on $Q_1$, consider a solution $u$ of
  \eqref{eq:main} in $Q_1$.  Then $u$ is $\alpha$-H\"older continuous
  on $Q_{\frac12}$ and
\begin{equation}\label{eq:holder-quad}
\sup_{\stackrel{x, y \in Q_{\frac12}}{ x \neq y}}
\frac{|u(x)-u(y)|}{|x-y|^\alpha} \le C_\alpha (\|u\|_{L^\infty (Q_1)}
+ \max(M_F,\|f\|_{L^n (Q_1)} + \gamma_F \|u\|_{L^\infty (Q_1)}))
\end{equation}
where $\alpha$ and $C_\alpha$ depend (only) on
$\|u\|_{L^\infty(Q_1)}$, $n$, $q$, $\lambda_F, \Lambda_F$, $\gamma_F$
and $\|\sigma\|_{L^q (Q_1)}$.
\end{cor}

\section{Applications: results of \cite{bd,dfq-cras}}
\label{sec:comparison}

In \cite{bd,dfq-cras}, Eq.~\eqref{eq homog} is considered. 
In \cite{bd}, $\alpha$ lies in $[0,1)$ and in \cite{dfq-cras},
$\alpha > -1$. They assume 
\paragraph{Assumption (H)}
\begin{itemize}
\item (H1) $F_0(tp,\mu X) = |t|^\alpha \mu F_0(p,X)$ for $t \neq 0$ and $\mu \ge 0$ for some $\alpha > -1$;
\item (H2) $|p|^\alpha \mathcal{M}^- (N) \le F_0(p, M+N) - F_0 (p,M) \le |p|^\alpha 
\mathcal{M}^+ (N)$.
\end{itemize}

The ABP estimate obtained in \cite{dfq-cras} is the following one
\begin{thm}[{\cite[Theorem~1]{dfq-cras}}]\label{thm:dfq}
Under Assumption~(H) and $c \ge 0$, any super-solution of \eqref{eq homog}
satisfies 
\begin{equation}\label{eq:dfq}
\sup_{B_d} u^- \le \sup_{\partial B_d} u^- + Cd 
\left( \left( \int_{B_d \cap \{ u + 
  M_\partial = \Gamma  (u)\}} (f_0^+)^n \right)^{1/n} \right)^{\frac1{1+\alpha}} 
\end{equation}
where $M_\partial= \sup_{\partial B_d} u^-$, $\Gamma (u) $ is the
convex hull of $\min (u+M_\partial,0)$ extended by $0$ on $B_{2d}$ and
$C$ is a constant (only) depending on $\|b\|_{L^n (B_d)}$, $n$,
$\alpha$ and $\|c\|_\infty$.
\end{thm}
D\'avila, Felmer and Quaas pointed out to us that it can be obtained from ours.
See below. 

The Harnack inequality obtained in \cite{bd} is the following one
\begin{thm}[{\cite[Theorems~3.1 and 3.2]{bd}}]\label{thm:bd}
Under Assumption~(H) and $c \ge 0$, any non-negative solution of \eqref{eq homog}
satisfies 
\begin{equation}\label{hibd}
\sup_{B} u \le C (\inf_{B} u + \| f_0 \|^{\frac1{1+\alpha}}_{L^n (B)}) \,.
\end{equation}
where $C$ is a constant (only) depending on $n$,
$q$, $\lambda_F, \Lambda_F$, $\|c\|_\infty$, $\alpha$    and $\|b\|_{L^q (Q_1)}$.
\end{thm}
\begin{rem}
This result is proved in \cite{bd} only in dimension $2$. 
Moreover, ours is slightly more precise since it depends on $q$ 
and $\|b\|_{L^q (Q_1)}$ instead of $\|b\|_{L^\infty (Q_1)}$.
\end{rem}
Their results are not included in ours but they can be derived with
little additional work. We mention that Birindelli and Demengel do not
prove this Harnack inequality  by proving first an ABP estimate. 

\paragraph{Proof of Theorem~\ref{thm:dfq}.}
D\'avila, Felmer and Quaas kindly explained to us the link between
their result and our result. We slightly adapt their argument
to get the general case. 

Assumption~(H2) implies $|p|^\alpha \mathcal{M}^- (X) \le F_0(p,X) \le
|p|^\alpha \mathcal{M}^+ (X)$.  

\begin{itemize}
\item If $\alpha \ge 0$,
\eqref{loclip2} and \eqref{loclip3} are satisfied for any $M_F>0$ with
$\sigma = |b|$, $f= \frac{f_0^+}{M_F^\alpha}$ and $\gamma_F u$ is
replaced with $c u |u|^\alpha$.  Moreover, \eqref{strictell} is
satisfied for any $M_F >0$ and with $\sigma =|b|$ and
$g(x,u)=\frac{(f_0(x) +cu|u|^\alpha)^+}{M_F^\alpha}$. In particular,
$g(x,-M_\partial) \le \frac{f_0^+(x)}{M_F^\alpha}$ since $c\ge
0$. Hence, our result gives
$$
\sup_{B_d} u^- \le \sup_{\partial B_d} u^- + Cd \left( M_F + \frac1{M_F^\alpha} 
\left(\int_{B_d \cap \{ u +
  M_\partial = \Gamma  (u)\}} (f_0^+)^n\right)^{1/n}\right) \, .
$$
Optimizing with respect to $M_F>0$ gives \eqref{eq:dfq}.
\item
If $\alpha = -\beta <0$, then $F(x,u,p,X) \ge 0$ and $u \le 0$ implies
$$
\mathcal{M}^+ (X) + |b(x)||p| + (f_0 + c u |u|^{-\beta})_+|p|^\beta \ge 0 \, . 
$$
Now using 
$$
g(p) =  |p|^{-\beta n} \left( |p|^{\frac{n(n-\beta)}{n-1}}+\mu^{\frac{n}{n-1}} \right)^{-n}
$$
in the proof of Theorem 1 permits to conclude after very similar computations.
\end{itemize}
\medskip

The Harnack inequality of \cite{bd} when $c=0$ can be
easily obtained from ours in any dimension (but not when $c \neq 0$).
The case $c \neq 0$ can also be treated but it requires to modify 
proofs.

\section{Proofs}
\label{sec:proofs}

\subsection{Proof of the weak Harnack inequality}

\begin{proof}[Proof of the weak Harnack inequality (Theorem~\ref{thm:whi})]
The proof of the weak Harnack inequality is performed in four steps. 
First, the problem is reduced to the case of a cube $Q$ of universal side-length  
(Lemma~\ref{lem:whi-reduced}), then it is proved that non-negative super-solutions
 can be bounded from above on $Q$ by a universal constant on a set of universal
positive measure (Lemma~\ref{lem:wh ub}). Next, the measures of all level sets 
of super-solutions (restricted to $Q$) are (universally) estimated from above. 
Finally, we prove the weak Harnack inequality in $Q$. 

\paragraph{Step 1.} As explained above, we first reduce the problem to a simpler one.
\begin{lem}[Reduction of the problem]\label{lem:whi-reduced}
Consider a non-negative super-solu\-tion $u$ of \eqref{eq:main} in $Q_{2R}$.
Then there exist universal constants $p_0$, $\eps_0$ and $C$ satisfying
\begin{equation}\label{eq:whi-reduced}
\left.\begin{array}{r}
 \inf_{Q_{3r}} u \le 1 \\
\max (M_F, \gamma_F, \|f\|_{L^n(Q_R)}, \|\sigma\|_{L^q(Q_{R})}) \le \eps_0
\end{array}\right\}  \Rightarrow \|u\|_{L^{p_0} (Q_r)} \le C \, .
\end{equation}
\end{lem}
We now explain how to derive the weak Harnack inequality from it.  Let
$v$ be a super-solution of \eqref{eq:main} in $Q_{R/t}$ for some
$t>0$.  We then define a function $v_s (y)= \frac{v(ty)}{V}$ with
$V>0$ and $t \in (0,1)$ to be chosen later.  Thanks to
Lemma~\ref{lem:rescaled} with $x_0=0$, $M_0 = V$, $R_0 = R /t$, the
new function $v_s$ satisfies $F_s\ge 0$ in $Q_{R}$ for a non-linearity
$F_s$ satisfying {\rm (A)} and \eqref{loclip2} with
$$
M_s = \frac{tM_F}{V} , \quad \gamma_s= \gamma_F t^2,
\quad \sigma_s (y)= t \sigma (ty), \quad f_s (y) = \frac{f^+(ty)}{V} \, .
$$
Hence, if one chooses 
\begin{eqnarray*}
V &=& \inf_{Q_{3r/t}} v + \delta + \eps_0^{-1} \max(M_F,  \|f\|_{L^n(Q_{R/t})}) \\
t&=& \left(\left(\frac{\|\sigma\|_{L^q (Q_{R/t})}}{\eps_0}\right)^{q/(q-n)} 
+ \left(\frac{\gamma_F}{\eps^0}\right)^{1/2} +1\right)^{-1}
\end{eqnarray*}
we obtain that $v$ satisfies
\begin{eqnarray*}
 \inf_{Q_{3r}} v_s \le 1 \\
\max (M_s, \gamma_s, \|f_s\|_{L^n(Q_{R})}, \|\sigma_s\|_{L^q(Q_{R})}) \le \eps_0 \, .
\end{eqnarray*}
We thus can apply Lemma~\ref{lem:whi-reduced} and we obtain from \eqref{eq:whi-reduced} 
the following estimate (after letting $\delta \to 0$)
\begin{equation}\label{eq:whi avant covering}
\|v\|_{L^{p_0} (Q_{r/t})} \le C ( \inf_{Q_{3r/t}} v + \max (M_F, \|f\|_{L^n(Q_{R/t})}) ) \, .
\end{equation}
A standard covering procedure  permits to  get \eqref{eq:whi}. 

\paragraph{Step 2.} In this step, we obtain a (universal) upper bound $M$ for non-negative
super-solutions in $Q_{R}$ on a set of (universal) positive measure $\mu$ if 
the $L^n$-norm of $f$ on $Q_{R}$, the $L^q$-norm of $\sigma$ on $Q_{R}$, $M_F$ and $\gamma_F$ 
are (universally) small. 
\begin{lem}[Upper bound on a subset 
of positive measure]\label{lem:wh ub}
There exist universal constants $r,R>0$, $\eps_0 >0$, $\mu \in (0,1)$ and $M_{\mathrm{B}}>0$ 
such that for any non-negative super-solution $u$ of \eqref{eq:main} in $Q_{R}$, we have
$$
\left.\begin{array}{r} \inf_{Q_{3r}} u \le 1 \\ 
    \max(M_F, \gamma_F  , \|f\|_{L^n (Q_{R})}, \|\sigma \|_{L^q(Q_{R})}) \le \eps_0
  \end{array} \right\}
\Rightarrow | \{ u \le M_{\mathrm{B}} \} \cap Q_r | \ge \mu |Q_r| \, . 
$$
\end{lem}
The proof of this lemma relies on the barrier function $\varphi$ that we constructed in the preliminary section
and on the ABP estimate applied to $w = u +\varphi$. 
\begin{proof}[Proof of Lemma~\ref{lem:wh ub}]
Given $\eps_0>0$ to be fixed later, we
consider $\varphi$ from Lemma~\ref{lem:barrier} and define $w= u + \varphi$. We
want to apply the ABP estimate (Theorem~\ref{thm:abp}) to the function $w$ on 
the ball $B_{R}$. 
\begin{itemize}
\item
First, $u \ge 0$ and $\varphi \ge 0$ on $\partial B_R$ 
hence $M_\partial = \sup_{\partial B_R} w^- = 0$.  
\item
Since $\inf_{Q_{3r}} u \le 1$ and $\varphi \le -2$ in $Q_{3r}$, we conclude
that $\inf_{Q_{3r}} w \le -1$; in other words, we have $\sup_{Q_{3r}} w^- \ge 1$. 
\item
We also claim that $w$ is a super-solution of an appropriate equation. 
More precisely, we claim that $w$ satisfies \eqref{eq:abp-structure} in $\{ w \le 0 \} \cap B_R$
for some appropriate continuous functions $\overline{f}$ and $\overline{\sigma}$.
\end{itemize}
Let us justify the last assertion and make precise what $\overline{f}$ and $\overline{\sigma}$ are. 
We write
\begin{eqnarray*}
0 & \le & F(x,u,Du,D^2 u) \\ 
&=& F(x,w-\varphi,Dw -D\varphi, D^2w - D^2 \varphi) \\
& \le & F(x,w + M_{\mathrm{B}},Dw -D\varphi, D^2 w - D^2 \varphi) \, .
\end{eqnarray*}
Assume next that $|Dw | \ge M_F+\eps_0 =:\overline{M}_F$, $D^2 w \ge 0$ (in the viscosity sense) 
and $w \le 0$. Then $|Dw-D\varphi| \ge M_F$ and we obtain from 
\eqref{loclip2} the following inequality
$$
0 \le \mathcal{M}^+ (D^2 w) - \mathcal{M}^- (D^2 \varphi) + \sigma |Dw| + \gamma_F M_B + \sigma \eps_0 + f  
$$
(we used the fact that $\mathcal{M}^+$ is subadditive and the relation between the two Pucci 
operators). Use next that $D^2 w \ge 0$ and 
$\varphi$ satisfies \eqref{barrier:semi concave}
$$
 \lambda_F \Delta w \le  \sigma |Dw| +C_{\mathrm{B}} \xi + \gamma_F M_B +  \sigma \eps_0 + f \, .
$$
Hence \eqref{eq:abp-structure} holds true with 

$$\overline{\sigma} = \sigma   \quad \text{ and } \quad
\overline{f} (x) = C_{\mathrm{B}} \xi + \gamma_F M_B  + \sigma \eps_0 + f \, .
$$
By using the ABP estimate for $w$ and the properties listed above 
satisfied by this function, we obtain
\begin{eqnarray*}
1 \le \sup_{B_R} w^- \le 3e^{C_{\mathrm{ABP}} (1+\|\sigma\|^n_{L^n(B_R)})} R 
 \left(\bar{M}_F + \left(\int_{\{\Gamma (w) = w\} \cap B_R} (\overline{f}^+)^n \right)^{1/n}\right)
\end{eqnarray*}
where $\Gamma (w)$ is the convex hull of $\min (w,0)$ after
extending $w$ to $B_{2R}$ by setting $w\equiv 0$ outside $B_R$. We now
use the fact that 
$$
\max(M_F ,\gamma_F, \|f\|_{L^n (Q_{R})}, \|\sigma
\|_{L^q(Q_{R})} ) \le \eps_0 \, ,$$
together with  
definitions of $\overline{f}$, $\overline{M}_F$
and the fact that $\mathrm{supp} \, \xi \subset Q_r$ in order to get
$$
1 \le 3e^{C_{\mathrm{ABP}} (1+(\omega_n R)^{n(1-\frac{n}q)} \eps_0^n)} R ( 3 \eps_0 +
(\omega_n R)^{1-\frac{n}q}\eps_0^2 + \eps_0 M_B + C_{\mathrm{B}} |\{\Gamma (w) = w\} \cap Q_r |
) \, .
$$
It is now enough to remark that
$$
\{ \Gamma (w) = w \} \subset \{ w \le 0 \} \subset\{ u \le M_{\mathrm{B}} \} 
$$ 
and to choose $\eps_0 \in (0,1)$ such that
$$
3 e^{C_{\mathrm{ABP}} (1+(\omega_n R)^{n(1 -\frac{n}q)} \eps_0^n)} R ( 3\eps_0
 + (\omega_n R)^{1-\frac{n}q} \eps_0^2+ \eps_0 M_B ) \le \frac12
$$
to conclude. We used here that $M_B$ is universal; in particular, it does not depend on $\eps_0$. 
\end{proof}

\paragraph{Step 3.} 
We derive from the previous lemma (Lemma~\ref{lem:wh ub}) an estimate
of any level set of super-solutions $u$ under consideration.  
Precisely, we use Lemma~\ref{lem:rescaled} together with the Calder\'on-Zygmund cube
decomposition lemma (see Lemma~\ref{lem:czd} in Appendix~\ref{app:additional})
in order to get the following result.  
\begin{lem}[Estimate of the measure of level sets  in $Q_r$]
\label{lem:wh level sets} 
Let $u$ be as in Lem\-ma~\ref{lem:wh ub}. Then there exist universal constants
$\eps>0$ and $d>0$ such that for all $t>0$, 
\begin{equation}\label{eq:wh level sets}
| \{ u \ge t\} \cap Q_r | \le d t^{-\eps} \, .
\end{equation}
\end{lem}
The proof of Lemma~4.6 in \cite{cc95}
can be easily adapted (with minor changes).
For the reader's convenience, a detailed proof is given in Appendix~\ref{app:additional}. 

\paragraph{Step 4.} We finally explain how to derive Lemma~\ref{lem:whi-reduced}. 
We first recall the following useful fact: if $u$ is a non-negative function, then
$$
\int_{Q_r} u^{p_0} = p_0 \int_0^{+\infty} t^{p_0-1} |\{ u \ge t\} \cap Q_r | dt \, .
$$
We can use the results of Lemmata~\ref{lem:wh ub} and \ref{lem:wh level sets}:
we thus choose $p_0 = \eps /2$ where $\eps$ appears in \eqref{eq:wh level sets} in
order to get
$$
\frac1{p_0} \int_{Q_r} u^{p_0} \le \int_0^1 t^{\eps/2-1} |Q_r| dt  + \int_1^{+\infty} t^{\eps/2-1} t^{-\eps} dt =: C \, .  
$$
This achieves the proof of Lemma~\ref{lem:whi-reduced} and  the proof of Theorem~\ref{thm:whi}. 
\end{proof}

\subsection{Proof of the local maximum principle}

The proof of the local maximum principle is easily adapted from \cite{cc95}. 
However, we give a detailed proof for the sake of completeness.

\begin{proof}[Proof of Theorem~\ref{thm:lmp}]
The proof is divided in two steps. First, the problem is reduced
to the case where the $L^\eps$-norm of $u$ is small; it
is to be proven that $u$ is bounded by a universal constant (Step 1).  
Then we explain how to get the universal bound (Steps 2 and 3). 

\paragraph{Step 1.} We state the lemma to be proven in Steps~2 and 3. 
\begin{lem}\label{lem:reduced2}
Consider a sub-solution $u$ of \eqref{eq:main} in $Q_{R}$. Then there exists
a universal constant $C>0$ such that
$$
\left.\begin{array}{r} 
\| u^+ \|_{L^\eps (Q_r)} \le d^{1/\eps} \\
\max(M_F,\gamma_F, \| f \|_{L^n (Q_{R})},\|\sigma\|_{L^q (Q_{R})}) \le \eps_0 
\end{array}\right\} \Rightarrow \sup_{Q_{\frac{r}4}} u \le C 
$$
where $\eps$ and $d$ appears in Lemma~\ref{lem:wh level sets}. 
\end{lem}
We now explain how to derive Theorem~\ref{thm:lmp} from this lemma. 
First, it is enough to get \eqref{eq:lmp} for a particular $p$ 
since the full result can be obtained by interpolation. In view of the previous
lemma, we consider $p = \eps$. By scaling $u$ and by using a covering argument,
we obtain the desired result.

\paragraph{Step 2.}
We remark that the assumption $\| u^+ \|_{L^\eps (Q_r)} \le
d^{1/\eps}$ implies for all $t>0$,
$$
|\{ u \ge t \} \cap Q_r| \le t^{-\eps} \int_{Q_r} (u^+)^\eps \le d t^{-\eps} \, .
$$
Remark that this estimate already appeared in the proof of the weak
Harnack inequality; see \eqref{eq:wh level sets} above. 
We next prove the following lemma.
\begin{lem} \label{lem:construction}
Consider a sub-solution $u$ of \eqref{eq:main} in $Q_{R}$ satisfying \eqref{eq:wh level sets}
and $F$ be such that 
$$
\max(M_F, \gamma_F, \|f\|_{L^n (Q_{R})},\|\sigma\|_{L^q (Q_{R})}) \le \eps_0 \, .
$$ 
Then there exists universal
constants $M_0 >1$ and $\Sigma >0$ such that
$$
\left. \begin{array}{r} 
x_0 \in Q_{\frac{r}2} , j \in \N \\
u (x_0) \ge \nu^{j-1} M_0 
\\
\end{array}\right\} \Rightarrow
\sup_{Q_{l_j} (x_0)} u > \nu^j M_0
$$
where $l_j=\Sigma \frac{M_0^{-\eps/n}}{\nu^{\eps j / n}} < \frac{r}2$ and $\nu = M_0 / (M_0 -1/2)>1$. 
\end{lem}
\begin{proof}[Proof of Lemma~\ref{lem:construction}]
We first choose $\Sigma$ and $M_0$ such that 
$$
\Sigma M_0^{-\eps /n} \le \frac{r}2 
$$
so that $l_j < \frac{r}2$ and $Q_{l_j} (x_0) \subset Q_{r}$. 
We now argue by contradiction by assuming that $\sup_{Q_{l_j} (x_0)} u \le \nu^j M_0$.
We have to exhibit a contradiction. 

On one hand, we have from \eqref{eq:wh level sets} and the fact that $r <R$ and $l_j <r/2$
\begin{equation}\label{eq:1}
| \{ u \ge \nu^j \frac{M_0}2 \} \cap Q_{\frac{l_j r}{R}} (x_0) | \le d \nu^{-j\eps} \left(\frac{M_0}2\right)^{-\eps}\, .
\end{equation}

On the other hand, since we have  $\sup_{Q_{l_j} (x_0)} u \le \nu^j M_0$ by assumption,
we can consider the following transformation
$$
T(y) = x_0 + \frac{l_j}{R} y 
$$
which defines a bijection between $Q_{R}$ and $Q_{l_j} (x_0)$. The function $v$ defined on $Q_{R}$ 
as follows
$$
v (y) = \frac{\nu M_0 - \frac{u(T(y))}{\nu^{j-1}}}{(\nu-1)M_0} \ge 0 
$$
thus satisfies $F_s(y,v,Dv,D^2v)=0$ in $Q_{R}$ with $F_s$ satisfying {\rm (A)} and \eqref{loclip2}
with 
\begin{eqnarray*}
M_s = \frac{t}{\nu^{j-1} (\nu-1)M_0} M_F, &\quad& \sigma_s (y) = t \sigma (x_0 +ty) , \\
\gamma_s = t^2 \gamma_F, &\quad&
f_s (y) =   \frac{t}{\nu^{j-1}(\nu-1)M_0} t f(x_0+ty)
\end{eqnarray*}
where $t = \frac{l_j}{R} < \frac12$. It is clear that $\gamma_s \le \gamma_F \le \eps_0$. 
Notice that 
$$
(\nu-1)M_0 = \frac{M_0}{2M_0 -1} > \frac12 > t
$$ 
hence $M_s \le M_F \le \eps_0$ and $f_s (y) \le t f(x_0 + ty)$.
We also have
\begin{eqnarray*}
\|\sigma_s \|_{L^q (Q_{R})} \le t^{1-\frac{n}q} \|\sigma\|_{L^q(Q_{R})} \le \eps_0 \\
\|f_s \|_{L^n (Q_{R})} \le  \|f\|_{L^n(Q_{R})} \le \eps_0 \, .
\end{eqnarray*}
Moreover, $v(0) = \frac{\nu M_0 -
  \frac{u(x_0)}{\nu^{j-1}}}{(\nu-1)M_0} \le 1$ by assumption on $u$;
thus $\inf_{Q_{3r}} v \le 1$. Hence, $v$ satisfies the assumptions of
Lemma~\ref{lem:wh ub} and we therefore obtain from Lemma~\ref{lem:wh
  level sets} the following estimate
$$
|\{ v \ge  M_0 \} \cap Q_r| \le d M_0^{-\eps} \, .
$$
We thus obtain 
\begin{equation}\label{eq:2}
|\{ u \le \nu^j \frac{M_0}2 \} \cap Q_{\frac{l_j r}{R}} (x_0) | \le 
\left(\frac{l_j}{R}\right)^n d M_0^{-\eps} \, .
\end{equation}

Combining \eqref{eq:1} and \eqref{eq:2}, we thus obtain 
$$
\left(\frac{l_j r}{R}\right)^n \le d \nu^{-j\eps}
\left(\frac{M_0}2\right)^{-\eps} + \left(\frac{l_j}{R}\right)^n d
M_0^{-\eps} \, .
$$
We also choose $M_0$ such that $d M_0^{-\eps} \le \frac{r^n}2$, and we obtain
$$
\frac12 \left(\frac{l_j r}{R}\right)^n \le d \nu^{-j\eps}
\left(\frac{M_0}2\right)^{-\eps} \, .
$$
Use now the definition of $l_j$ and get
$$
\frac12 \left(\frac{\Sigma r}{R}\right)^n \le d2^{\eps} \, .
$$
We next choose $\Sigma > d^{\frac1n} 2^{\frac{\eps+1}n} \frac{R}r$ in order to get a contradiction.
\end{proof}

\paragraph{Step 3.}
We prove Lemma~\ref{lem:reduced2}. By Step~2, we know that the
sub-solution $u$ satisfies the conclusion of
Lemma~\ref{lem:construction}. In particular, the series $\sum_j l_j$
converges and we can find a universal integer $j_0 \ge 1$ such that
$\sum_{j \ge j_0} l_j \le \frac{r}8$.

We now claim that $\sup_{Q_{\frac{r}4}} u \le \nu^{j_0-1} M_0$. We
argue by contradiction by assuming that this is not true and by
exhibiting a contradiction. Let us assume that there exists $x_{j_0}
\in Q_{\frac{r}4}$ such that $u (x_{j_0}) \ge \nu^{j_0-1} M_0$. Hence,
we can apply Lemma~\ref{lem:construction} and we get a point
$x_{j_0+1}$ such that $|x_{j_0+1} - x_{j_0}|_\infty \le l_{j_0}/2$ and
$u(x_{j_0+1}) \ge \nu^{j_0} M_0$.  By induction, we construct a
sequence $(x_j)_{j \ge j_0}$ such that $|x_{j+1} -x_j| \le l_j/2$ and
$u(x_{j+1}) \ge \nu^j M_0$ as long as $x_j \in Q_{\frac{r}2}$. This is
always the case since
$$
|x_j|_\infty \le |x_{j_0}|_\infty + \sum_{k=j_0}^{j-1} |x_{k+1}-x_k|
\le \frac{r}8 + \frac{r}8 \le \frac{r}4 \, .
$$
We now get a contradiction since $u$ is upper semi-continuous; indeed,
it is bounded from above in $Q_{\frac{r}2}$ so it cannot satisfy
$u(x_{j+1}) \ge \nu^j M_0$ for all $j \ge j_0$.  The proof is now
complete.
\end{proof}

\subsection{Proofs of Theorems~\ref{thm:wh-quad} and \ref{thm:lmp-quad}}

\begin{proof}[Proofs of Theorems~\ref{thm:wh-quad} and \ref{thm:lmp-quad}]
Both proofs rely on a transform of Cole-Hopf type in order to remove quadratic terms. 

In order to understand why the exponential change of variables is the
right one, we consider $v= h^{-1} (u)$ for some increasing convex
function $h$ and we remark that $v$ satisfies
$$ 
\mathcal{M}^+ (D^2 v) + \sigma (x) |Dv| + \frac{f^+(x)}{h'(v)} \ge 0
$$ 
as soon as $h$ satisfies $\lambda_F h'' - \sigma_2 (h')^2 = 0$. 
We thus choose 
$$
h(t) = \frac{\lambda_F}{\sigma_2} \ln \left( 1 - \frac{\sigma_2 t}{\lambda_F} \right)^{-1} \, .
$$
We thus derive \eqref{eq:whi-quad} from \eqref{eq:whi} by remarking that
$$
\frac{ 1 - e^{-\frac{\sigma_2 \|u\|_{L^\infty (Q_1)} }{\lambda_F}}}%
{\frac{\sigma_2 \|u\|_{L^\infty (Q_1)} }{\lambda_F}} u \le v \le u 
$$
and $\frac1{h'(t)} = 1 - \frac{\sigma_2 t}{\lambda_F} \le 1$. 

We proceed in the same way in order to prove Theorem~\ref{thm:lmp-quad}. Remark that
we can assume without loss of generality that the solution is non-negative. 
\end{proof}

\appendix

\section{Additional proofs}
\label{app:additional}

\subsection{Proofs of Lemmata~\ref{lem:1} and \ref{lem:2}}

In this paragraph, we explain how to prove Lemmata~\ref{lem:1} and
\ref{lem:2} by adapting the techniques of \cite{convexe}.

We first recall useful facts from convex analysis. The first one deals
with the convex hull $U^{**}$ of a function $U$.
\begin{prop}\label{prop:combinaison}
  Let $\Omega$ be a bounded convex open set and $U :\overline{\Omega}
  \to \R$ be lsc. For $x\in \overline{\Omega}$, consider $(p,A) \in
  D^{2,-} U^{**} (x)$. There then exist $x_1, \dots, x_q \in
  \overline{\Omega}$, $q \le n$, $\lambda_1, \dots, \lambda_q \in
  (0,1]$, $\sum_{i=1}^q \lambda_i =1$ such that
\begin{equation}\label{combinaison}
  \left\{ \begin{array}{l}
      x = \sum_{i=1}^q \lambda_i x_i  \, \\
      U^{**}(x) = \sum_{i=1}^q \lambda_i U(x_i) \, .
\end{array}\right.
\end{equation}
Moreover $U^{**}$ is linear on the convex hull of $\{x_1, \dots,
x_q\}$.  In particular, $A \le 0$ for a.e. $x \in \{ U = U^{**} \}$.
\end{prop}
We next recall a result from \cite{convexe} (see also \cite{all97})
about the subjet of the convex hull $U^{**}$ of a function $U$.
\begin{prop}[{\cite[Proposition~3]{convexe}}] \label{sousjet} Let
  $\Omega$ be a bounded convex open set and $U :\overline{\Omega} \to
  \R$ be lower semi-continuous. For $x\in \overline{\Omega}$, consider
  $(p,A) \in D^{2,-} U^{**} (x)$. Consider $x_i$ and $\lambda_i$ such
  that \eqref{combinaison} hold true.  Then for every $\eps >0$, there
  are $A_i \in \S$, $i=1,\dots, q$, such that
\begin{equation}\label{sousdiffenv}
  \left\{ \begin{array}{l}
      (p, A_i) \in \overline{D}^{2,-} U (x_i), \\
      A_\eps \le \Box_{i=1}^q \left(  \lambda_i^{-1} A_i\right) 
\end{array}\right.
\end{equation}
where $\Box$ denotes the parallel sum of matrices. We recall that
$$
(A \Box B ) \xi \cdot \xi = \inf_{\zeta \in \R^n} \{ A (\xi-\zeta)
\cdot (\xi-\zeta) + B \zeta \cdot \zeta \} \, .
$$
\end{prop}
We next recall a (necessary and) sufficient condition for a function
to be semi-concave.
\begin{lem}[{\cite[Lemma~1]{all97}}]\label{lem:cs semi concave}
  Consider a bounded convex open set $\Omega$ and $U: \Omega \to \R$ a
  lower semi-continuous function. Assume that there exists $C>0$ such
  that for all $x \in \Omega$ and all $(p,A) \in D^{2,+} U (x)$, $A
  \le C I$. Then $U- C|\cdot|^2 /2$ is concave.
\end{lem}
We finally recall a useful approximation lemma from \cite{all97}.
\begin{lem}[\cite{all97}]\label{lem:non-negative hessians}
  Consider a convex set $\Omega$ and a convex function $V: \Omega \to
  \R$.  For all $(p,A) \in D^{2,-} V (x)$, there exists $(x_n)_n$ and
  $(p_n,A_n) \in D^{2,-} V (x_n)$ such that $x_n \to x$, $p_n \to p$,
  $A_n \ge 0$ and $A \le A_n + \frac1n$.
\end{lem}

We now turn to the proofs of the two lemmata.
\begin{proof}[Proofs of Lemmata~\ref{lem:1} and \ref{lem:2}]
  The function $v = u+M_\partial$ is a super-solution of
$$
G(x,v,Dv, D^2 v)=0
$$ 
with $G(x,r,p,X)=F(x,r+M_\partial,p,X)$. Then $\Gamma (u)$ is the
convex hull of the function $\min (v,0)$.

We first reduce the problem to the study of subjet of the function
$\Gamma (u)$.
\begin{lem}\label{lem:reduction}
  Assume that $\Gamma (u)$ satisfies the following properties
\begin{eqnarray}\label{eq:semi concave}
  \exists C>0 \; / \; \forall x \in \mathcal{B},\forall (p,A) \in  D^{2,-} \Gamma (u) (x), A \le C I \, , 
  \\
  \label{eq:hessian estimate}
  \left\{ \begin{array}{r}
      \forall x \in \mathcal{B} \cap \{ \Gamma(u) = u + M_\partial\} , \forall (p,A) \in  D^{2,-} \Gamma (u) (x), \\
      A \le \lambda_F^{-1} (\sigma (x) |p| + f^+ (x) ) I \, , 
    \end{array} \right.
  \\
  \label{eq: linear}
\Gamma (u) \text{ is linear on } \mathcal{B} \setminus \{ x \in B_d : \Gamma(u) = u + M_\partial \} \, .
\end{eqnarray}
Then $\Gamma (u)$ satisfies conclusions of Lemmata~\ref{lem:1} and
\ref{lem:2}.
\end{lem}
\begin{proof}
  Thanks to Lemma~\ref{lem:cs semi concave}, Eq.~\eqref{eq:semi
    concave} implies that $\Gamma (u)$ is semi-concave in
  $\mathcal{B}$.  Since $\Gamma (u)$ is convex, this implies that
  $\Gamma (u)$ is $C^{1,1}$ in $\mathcal{B}$.  Hence Lemma~\ref{lem:1}
  is proved. We next remark that \eqref{eq: linear} implies Point~1 in
  Lemma~\ref{lem:2}.  Eventually, \eqref{eq:hessian estimate} together
  with Alexandroff theorem permits to get Point~2.  We recall that
  Alexandroff theorem implies that a convex function is almost every
  twice differentiable.  Hence the proof of Lemma~\ref{lem:reduction}
  is now complete.
\end{proof}
We now prove the following lemma in order to achieve the proof of
Lemmata~\ref{lem:1} and \ref{lem:2}.
\begin{lem}\label{lem:core} 
  The function $\Gamma (u)$ satisfies \eqref{eq:semi concave},
  \eqref{eq:hessian estimate} and \eqref{eq: linear}.
\end{lem}
\begin{proof}
  We first remark that \eqref{eq: linear} is a consequence of
  Proposition~\ref{prop:combinaison} and of Alexandroff theorem.

  We now turn to the proof of \eqref{eq:semi concave} and
  \eqref{eq:hessian estimate}.  Consider next $x \in \mathcal{B}$ and
  $(p,A) \in D^{2,-} \Gamma(u) (x)$. Notice that we cannot just prove
  \eqref{eq:semi concave} for a.e. $x \in \mathcal{B}$.  In view of
  the definition of $\mathcal{B}$ (see also Remark~\ref{rem:definition
    de B}), we know that $|p| \ge M_F$.  Thanks to
  Lemma~\ref{lem:non-negative hessians}, we can assume without loss of
  generality that $A \ge 0$. We now distinguish two cases.

  \paragraph{Case 1: $x \in \mathcal{B} \cap \{\Gamma (u) =u+
    M_\partial\}$.}

  In such a case, $(p,A) \in D^{2,-} \Gamma(u) (x) = D^{2,-} u (x)$,
  and since $|p| \ge M_F$, we have $F(x,u(x),p,A) \ge 0$. Now
  \eqref{strictell} yields
$$
- \lambda_F \mathrm{Tr} A + \sigma (x) |p| + f^+(x) \ge 0
$$
and since $A\ge0$, we conclude that \eqref{eq:hessian estimate} holds
true and the right hand side is bounded in $B_d$ since $\Gamma (u)$ is
Lipschitz continuous and $\sigma$ and $f^+$ are continuous.  \medskip

Remark that the previous inequality also
holds true for $A$ such that $(p,A) \in \bar{D}^{2,-} \Gamma (u) (x)$, $A \ge 0$, since the
equation is also satisfied for limiting semi-jets.

\paragraph{Case 2: $x \in \mathcal{B} \setminus \{ \Gamma(u) = u +
  M_\partial \}$.}
 
There then exist $x_i \in \bar{B}_d$ and $\lambda_i \in (0,1]$, $i
=1,\dots, q$, such that \eqref{combinaison} holds true (where $U=u+
M_\partial$).  We know that there is at most one point $x_i$ on
$\partial B_{2d}$ and the others are in $B_d$; if not, $\Gamma (u)
\equiv 0$ and there is nothing to prove. Moreover, $x_i \in
\mathcal{B}$ for $i=1,\dots, q$.

By Proposition~\ref{sousjet}, for any $\eps >0$, there exist $q$
matrices $\lambda_i^{-1} A_i \ge A_\eps \ge 0$ such that $\Box_{i=1}^q
\lambda_i^{-1} A_i \ge A_\eps$ and $(p,A_i) \in \bar{D}^{2,-} u (x_i) $.

If there are no points on $\partial B_{2d}$, we deduce from Case~1
that for all $i$, $A_i \le C I$ and $A_\eps \le CI$ follows.

If $x_q \in \partial B_{2d}$, say, then we deduce from
\eqref{combinaison} that $\lambda_q \le 2/3$; hence, there exists $i
\in \{1, \dots, q-1\}$ such that $\lambda_i \ge 1/3n$. For instance
$i=1$.  Then we conclude that
$$
A_\eps \le \frac{1}{\lambda_1} A_1 \le 3n C I. 
$$ 
Passing to the limit on $\eps$, we obtain $A \le C I$ (for some new
constant $C$).
\end{proof}
\end{proof}

\subsection {Proof of Lemma~\ref{lem:wh level sets}}

In order to prove Lemma~\ref{lem:wh level sets}, we need 
the Calder\'on-Zygmund cube decomposition such as stated
in \cite{cc95}. We thus first recall it. 
We use notation from \cite{cc95}.  Given $r>0$, the cube $Q_r$ is
split in $2^n$ cubes of half side-length.  We do the same with all the
new cubes and we iterate the process. The cubes obtained in this way
are called \emph{dyadic cubes}.
If $Q$ is a dyadic cube of $Q_r$, $\tilde{Q}$ denotes a dyadic cube
such that $Q$ is one of $2^n$ cubes obtained from $\tilde{Q}$.
\begin{lem}[Cube decomposition]\label{lem:czd}
Consider $r>0$ and two measurable subsets $A \subset B \subset Q_r$.  
Consider $\delta \in (0,1)$ such that
\begin{itemize}
\item $|A| \le \delta |Q_r|$ ;
\item if $Q$ is a dyadic cube of $Q_r$ such that $|A \cap Q| > \delta |Q|$,
then $\tilde{Q} \subset B$.
\end{itemize}
Then $|A | \le \delta |B|$. 
\end{lem}
As far as the proof of this lemma is concerned, the reader is referred to \cite{cc95}.
We now turn to the proof of Lemma~\ref{lem:wh level sets}. 

\begin{proof}[Proof of Lemma~\ref{lem:wh level sets}]
We are going to prove the following estimate
\begin{equation}\label{eq:wh ls discrete}
|\{ u \ge (M_B)^k \} \cap Q_r | \le (1-\mu)^k |Q_r|
\end{equation}
where $M_B$ and $\mu$ are given by Lemma~\ref{lem:wh ub}.  The reader
can check that \eqref{eq:wh level sets} derives from \eqref{eq:wh ls
  discrete} with $d=(1-\mu)^{-1}$ and $\eps = - \ln (1-\mu) / \ln
M_B$.

We prove \eqref{eq:wh ls discrete} by induction. 
Lemma~\ref{lem:wh ub} implies that \eqref{eq:wh ls discrete} holds
for $k=1$. We now consider $k \ge 2$, we assume that \eqref{eq:wh ls discrete}
holds for $k-1$ and we prove it for $k$. To do so, we are going to
apply Lemma~\ref{lem:czd} with the two following sets $A \subset B \subset Q_r$
\begin{eqnarray*}
A &=& \{ u > (M_B)^k \} \cap Q_r \, , \\
B &=& \{ u > (M_B)^{k-1} \} \cap Q_r 
\end{eqnarray*}
and with $\delta = 1-\mu$. Remark that $A \subset \{ u > M_B \} \cap Q_r$;
hence $|A|\le (1-\mu) |Q_r|$. It thus remains to prove that if $Q$ is
a dyadic cube of $Q_r$ such that 
\begin{equation}\label{assum}
|A \cap Q| > (1-\mu) |Q|
\end{equation}
 then
the predecessor $\tilde{Q}$ of $Q$ satisfies $\tilde{Q} \subset B$. 
Consider such a dyadic cube $Q = Q_{\frac{r}{2^i}} (x_0)$ and suppose that 
$\tilde{Q}$ is not contained in $B$. Then there exists $\tilde{x} \in \tilde{Q}$
such that $u(\tilde{x}) \le (M_B)^{k-1}$. We now use Lemma~\ref{lem:rescaled}
with $R_0 = R$, $t_0 = \frac{1}{2^i}$ and $M_0 = (M_B)^{k-1}$ 
to get a rescaled function $u_s$ satisfying $F_s=0$ with $F_s$ such that 
\eqref{loclip2} holds with constants $M_s \le M_F$, $\gamma_s \le \gamma_F$  
and functions $f_s$, $\sigma_s$ 
satisfying $\|f_s\|_{L^n (Q_{R})} \le \|f \|_{L^n (Q_{R})}$ and 
$\|\sigma_s\|_{L^q (Q_{R})} \le \|\sigma \|_{L^q (Q_{R})}$. We thus can apply
Lemma~\ref{lem:wh ub} if $\inf_{Q_{3r}} u_s \le 1$. This is indeed the case
$$
\inf_{Q_{3r}} u_s \le \frac{u(\tilde{x})}{(M_B)^{k-1}} \le 1 \, .
$$
Hence, $|Q \setminus A| > (1-\mu) |Q|$ which contradicts \eqref{assum}. 
\end{proof}

\subsection{Proof of Corollary~\ref{cor:holder}}

\begin{proof}[Proof of Corollary~\ref{cor:holder}]
We use the notation of \cite{cc95}: for all $r \in (0,1)$, $m_r =
\inf_{Q_r} u$, $M_r = \sup_{Q_r} u$, $o_r = M_r - m_r =
\mathrm{osc}_{Q_r} u$. The non-negative functions $u-m_1$ and $M_1 -u$
satisfy equations $F_- =0$, $F^+=0$ respectively for some non-linearities
$F_-$ and $F^+$ satisfying \eqref{loclip2}, \eqref{loclip3} with $f$ 
replaced with $f+\gamma_F M_1$. Hence, we can apply the Harnack
inequality two $M_1 -u$ and $u-m_1$ and get
\begin{eqnarray*}
  M_{1/2}-m_1 \le C ( m_{1/2} -m_1 + \max(M_F, \|f\|_{L^n (Q_1)}+ \gamma_F |m_1|)) \, , \\
  M_1-m_{1/2} \le C ( M_1 -M_{1/2} + \max(M_F, \|f\|_{L^n (Q_1)}+ \gamma_F |M_1|)) 
\end{eqnarray*}
where we can assume without loss of generality that $C>1$.  Adding
these two inequalities and rearranging terms, we obtain
$$
\mathrm{osc}_{Q_{1/2}} u \le \frac{C-1}{C+1} \mathrm{osc}_{Q_1} u + 2
\max(M_F, \|f\|_{L^n (Q_1)}+ \gamma_F \|u\|_{L^\infty (Q_1)}) \, .
$$
We now use Lemma~8.3 in \cite{gt} in order to get \eqref{eq:holder}.
\end{proof}

\bibliographystyle{siam}

\def\cprime{$'$} \def\polhk#1{\setbox0=\hbox{#1}{\ooalign{\hidewidth
  \lower1.5ex\hbox{`}\hidewidth\crcr\unhbox0}}}

\noindent \textsc{Cyril Imbert} \medskip\\
\textsc{Universit\'e Paris-Dauphine} \\
\textsc{ CEREMADE, UMR CNRS 7534 }\\
 \textsc{   place de Lattre de Tassigny} \\
\textsc{ 75775 Paris cedex 16, France} \\
\texttt{imbert@ceremade.dauphine.fr}

\end{document}